\documentclass[preprint,12pt]{elsarticle}
\journal{Linear Algebra and its Applications}

\usepackage{amsmath}
\usepackage{mathdots}
\usepackage{amsfonts}
\usepackage{amssymb}
\usepackage{latexsym}
\usepackage{graphicx}
\usepackage{pgfplots}
\usepackage{todonotes}
\usepackage{parskip}
\usepackage{bm}
\usepackage{subdepth} %fix to weird spacing of "V_+^T" vs "V_+"

\usepackage{color}
\usepackage{fullpage}
\usepackage{verbatim}

\def\Lh{\widehat\Lambda}

\def\Lt{\smash{\widetilde\Lambda}}
\def\Lh{\smash{\widehat\Lambda}}

\def\Vh{\smash{\widehat V}}
\def\Vt{\smash{\widetilde V}}
\def\Ah{\smash{\widehat A}}
\def\Rh{\smash{\widehat R}}
\def\At{\smash{\widetilde A}}
\def\Dt{\smash{\widetilde D}}
\def\gap{\mbox{gap}}

\newcommand{\project}[1]{{\bm\Pi}_{#1}}
\newcommand{\Sym}{\mbb{S}}
\newcommand{\mbb}{\mathbb}

\newcommand{\eqdef}{:=}
\newcommand{\tpose}{^T}
\newcommand{\norm}[1]{\left\|{#1}\right\|}

\newenvironment{proof}{\paragraph{Proof:}}{\hfill$\square$}

\newcommand{\diag}{{\rm diag}}

\usepackage{xcolor,calc}

\newtheorem{theorem}{Theorem}[section]
\newtheorem{lemma}{Lemma}[section]

\newtheorem{corollary}{Corollary}[section]

\newtheorem{example}{Example}[section]

\newcommand{\rr}[1]{\textcolor{red}{#1}}

\newcommand{\ignore}[1]{}

\begin{document}

\begin{frontmatter}



  \title{Accuracy of approximate projection to the semidefinite cone}

  \author[engineering]{Paul J. Goulart}
  \ead{paul.goulart@eng.ox.ac.uk}
  \author[maths]{Yuji Nakatsukasa}
  \ead{nakatsukasa@maths.ox.ac.uk}
  \author[engineering]{Nikitas Rontsis\corref{cor}}
  \ead{nikitas.rontsis@eng.ox.ac.uk}
  \cortext[cor]{Corresponding author}
  \address[engineering]{Department of Engineering Science,
  University of Oxford,
  Oxford, OX1 3PN, UK.}
  \address[maths]{Mathematical Institute,  University of Oxford, Oxford, OX2 6GG, UK, and\\ National Institute of Informatics, Japan.}

  \begin{abstract}
    When a projection of a symmetric or Hermitian matrix to the positive semidefinite cone is computed approximately (or to working precision on a computer), a natural question is to quantify its accuracy.
    A straightforward bound invoking standard eigenvalue perturbation theory
     (e.g.~Davis-Kahan and Weyl bounds) suggests that the accuracy would be inversely proportional to the spectral gap, implying it can be poor in the presence of small eigenvalues.
    This work shows that a small gap is not a concern for projection onto the semidefinite cone, by deriving error bounds that are gap-independent.
  \end{abstract}



  \begin{keyword}
    positive semidefinite cone \sep
    projection \sep
    matrix nearness problem \sep
    eigenvalue perturbation theory
    % https://mathscinet.ams.org/mathscinet/msc/msc2010.html
    \MSC[2010] 65F15 \sep %	Numerical analysis: Eigenvalues, eigenvectors
    15A45 \sep % Linear	and multilinear algebra; matrix theory: Miscellaneous inequalities involving matrices
    15B48 \sep %	Positive matrices and their generalizations; cones of matrices
    90C22 % Semidefinite programming



  \end{keyword}

\end{frontmatter}

\section{Introduction}
Projection of symmetric or Hermitian matrices to the positive semidefinite cone is a standard operation that arises frequently in scientific computing. A common, practical, example is restoring positive definiteness of partially unknown or corrupted correlation matrices~\cite{Qi2006} arising in e.g., economics~\cite{Higham2002}, integrated circuit design~\cite{Xiong2007} and wireless communications ~\cite{Higham2016}. Further, more generic, examples include quasi-newton optimization methods \cite[\S 4.2.2]{Fletcher1987}, incomplete matrix factorizations of sparse matrices~\cite[\S 15.11]{Duff2017} and, finally, first order methods for solving semidefinite problems (SDPs)~\cite{Boyd2011,Vandenberghe1996} which, as we proceed to explain, was the motivating example for this work.

The projection operator $\project{+}$ maps a symmetric matrix to a nearest symmetric positive semidefinite matrix. As such, it belongs to the class of \emph{matrix nearness problems}, a survey of which can be found in \cite{Higham1989}. When a unitarily invariant norm is chosen as a distance metric,
%\footnote{Note that the projection is not unique when the spectral norm is chosen as a distance metric.}
the projection operator can be expressed in ``closed form''. Indeed, we show in \S\ref{sec:other_norms} that if $A\in\mathbb{S}^{n}$ with eigenvalue decomposition
\begin{equation} \label{eqn:eigendecomposition}
  A = \begin{bmatrix}
    V_{+} & V_{-}
  \end{bmatrix}
  \begin{bmatrix}
  \Lambda_{+} \\
  & \Lambda_{-}
\end{bmatrix}
\begin{bmatrix}
  V_{+} & V_{-}
\end{bmatrix}^T,
\end{equation}
where $[V_{+} \; V_{-}]$ is an $n \times n$ orthogonal matrix with
$V_+$ having $k$ columns and $\Lambda_{+}, \Lambda_-$ diagonal matrices containing the positive and non-positive eigenvalues respectively, then we can set: %\footnote{For a given $A$ this is the unique projection to the semidefinite cone in the Frobenius norm. In the spectral norm, there might exist many positive (semi)definite matrices $X$ that minimize $\norm{X - A}_2$ with $V_+\Lambda_+V_+^T$ always being a minimizer .}
\begin{equation} \label{eqn:projection}
  \project{+}(A) \eqdef V_+\Lambda_+V_+^T.
\end{equation}

When $\project{+}(A)$ is computed in practice, either via the full eigenvalue decomposition or an iterative method such as the (block) Lanczos method, one obtains approximations
$\Lh$ ($\tilde k \times \tilde k$ diagonal), $\Vh$ ($n \times \tilde k$ orthonormal) to $\Lambda_+, V_+$
such that
$\Vh\Lh\Vh^T\approx V_+\Lambda_+V_+^T$.
We usually (but not necessarily) take $\Lh = \Vh^TA\Vh$, which holds by the
% as is natural and common; for example it always holds with the
standard Rayleigh-Ritz process~\cite[Ch.~11]{Parlett1998}.
Note that the sizes $k$ and $\tilde k$ are not assumed to agree since the number of positive eigenvalues computed may not be exact, especially in the presence of eigenvalues close to 0.

With a backward stable algorithm, we have
$\|A\Vh-\Vh\Lh\| = O(u)\|A\|$, where $u$ is the unit roundoff.
%In general, $\|A\Vh-\Vh\Lh\|$ along with $\Lh$ are the only information available after computation;
The question we address is: what does this imply
about the size of the projection error $\|\Vh\Lh\Vh^T- V_+\Lambda_+V_+^T\|$?

The first observation can arise if we assume that $k = \tilde k$ and define
$\Vh=V_++\Delta V$, $\Lh=\Lambda_++\Delta \Lambda$, resulting in
\begin{equation}
  \label{eq:vtvnaive}
\|\Vh\Lh\Vh^T- V_+\Lambda_+V_+^T\|
\leq 2\|\Delta V\Lambda_+V_+^T\|
+\| V_+\Delta \Lambda V_+^T\|
+O(\norm{\Delta}^2)
\end{equation}
where $\norm{\Delta} \eqdef \max(\norm{\Delta V}, \norm{\Delta \Lambda})$. This suggests bounding the term $\|\Delta V\|$, which usually dominates the $\|\Delta \Lambda\|$ term since eigenvalues of symmetric matrices are always well-conditioned and hence $\|\Delta \Lambda\|=O(u)\|A\|$ with a backward stable method.   The term $\Delta V$ is the error in the computed eigenvectors, for which
the classical Davis-Kahan $\sin\theta$ theorem~\cite{Davis1970} shows that
\begin{equation}
  \label{eq:DK}
  \|\sin\angle(V_+,\Vh)\|\leq \frac{\|R\|}{\gap}
  %\frac{\|A\Vh-\Vh\Lh\|_2}{\gap}
  %=\frac{\|A\Vh-\Vh\Lh\|_2}{\gap}
\end{equation}
in any unitarily invariant norm, where $\angle(V_+,\Vh)$ is a $k \times k$ diagonal matrix containing the principal angles between $V_+$ and $\Vh$. Here $R \eqdef A\Vh-\Vh\Lh$ is the residual and
\mbox{`gap'} is the distance between the computed positive eigenvalues $\Lh$ and the exact  nonpositive eigenvalues, i.e.
\begin{equation}
  \label{eqn:gap}
  \gap \eqdef \min_{i, j} \left| \Lh_{(i, i)} - {\Lambda_-}_{(j, j)}\right|.
\end{equation}% YN: subtlety here but removed by 'same size' assumption.
%eigenvalues of $V_+$(which are roughly the nonpositive eigenvalues).
The bound~\eqref{eq:DK} is essentially sharp~(although sometimes improvable~\cite{Nakatsukasa2018}), and implies\footnote{
  $\norm{\Delta V}$ also depends on the ``small gap'' $\eqdef \min_{i \neq j} \left| \Lh_{(i, i)} - {\Lambda}_{(j, j)}\right|$, where $\Lambda \eqdef \big[\begin{smallmatrix}  \Lambda_{+} \\ &   \Lambda_{-} \end{smallmatrix}\big]$, but this is not a practical concern when $\norm{R} \leq $ ``small gap''. See \cite[Theorem 3.1 and Remark 3.1]{Nakatsukasa2018} for a rigorous discussion on the effect of spectral gaps on $\norm{\Delta V}$.
} $\|\Delta V\|\lesssim \frac{\|R\|}{\gap}$.
Together with~\eqref{eq:vtvnaive}, we obtain
\begin{equation}  \label{eq:naive}
\|\Vh\Lh\Vh^T- V_+\Lambda_+V_+^T\|
\lesssim \frac{2\|R\|\|\Lambda_+ V_+^T\|}{\gap} \leq \frac{2\|R\|\| A\|}{\gap}.
\end{equation}
The problem here is that \eqref{eq:vtvnaive} suggests that %$\|\Delta V\|$ is
the projection error would be
 large if gap is small, and
we have no control over how small $\gap$ can be! Indeed, in many interesting cases $A$ has eigenvalues of small magnitude (both positive and negative). This causes problems in multiple ways---the $\gap$ indeed gets small, and
since approximations to the small eigenvalues may have the wrong signs,
%positive eigenvalues may be computed negatively
$\Lh$ may not contain the correct number of positive eigenvalues.
%Intuitively, this appears an unpleasant situation--the computed projection resulted in a matrix of wrong rank.

The situation is further exacerbated when $\Lh$, $\Vh$ are computed not to full working precision, but only to a looser tolerance $\epsilon\gg u$. Then, $\|R\|=O(\epsilon\|A\|)$ instead of $O(u\|A\|)$.
Indeed, this work was motivated in the context of devising efficient algorithms for semidefinite optimization with the Alternating Direction Method of Multipliers (ADMM) \cite{Garstka2019}. ADMM is an iterative method in which every iteration entails a projection to the semidefinite cone.  These projections typically dominate the (total) computation time of ADMM, thus an inexact method for their computation is desirable to reduce ADMM's execution time and scale ADMM to large SDPs. Naturally, the projection error must be quantified and controlled so that ADMM maintains its convergence properties.  This can be achieved, for example, when the projection errors are summable, i.e.~when the sum of the projection errors over all the ADMM iterations is bounded \cite[Theorem 8]{Eckstein1992}. However, since the gap is unknown and cannot be controlled, bounds like~\eqref{eq:naive} are not very useful in such situations.

The purpose of this paper is to show that, fortunately, these problems suggested by~\eqref{eq:naive} are not a concern, i.e., small gaps do not affect the projection accuracy.
Specifically, our main result is
\begin{equation}  \label{eq:mainsimple}
\|\Vh\Lh\Vh^T- V_+\Lambda_+{V_+^T}\|_F\leq \sqrt{2}\|R\|_F  ,
\end{equation}
which holds when $\Lh$ has the same size as $\Lambda_+$ (otherwise the bound worsens, but only slightly). Unlike~\eqref{eq:naive}, the bound~\eqref{eq:mainsimple} is sharp up to a constant smaller than $\sqrt{2}$.
Since \mbox{$\|R\|_F=\|A\Vh-\Vh\Lh\|_F$} is easy to compute,~\eqref{eq:mainsimple} provides a practical means to estimate the projection accuracy.
%, using~\eqref{eq:mainsimple}
% we can bound

Noting that
$\project{+}(A) = V_+\Lambda_+V_+^T = A(V_+V_+^T)$,
we also treat an alternative measure of the projection accuracy
\begin{equation}  \label{eq:alte}
\|A(V_+V_+^T-\Vh\Vh^T)\|,
\end{equation}
and prove similar bounds for this quantity.
The two quantities $\|\Vh\Lh\Vh^T - V_+\Lambda_+V_+^T\|$ and~\eqref{eq:alte} are closely related, and
 we show in Section~\ref{sec:main2} that they must lie within $\|R\|$ of each other.

Here is an intuitive explanation for the gap-independence (which is easier to see with~\eqref{eq:alte}): while $\|\Delta V\|$ does depend on $1/\gap$, large errors in $\Delta V$ lie only in directions of eigenvectors $v_i$ of $A$ with small corresponding eigenvalues $\lambda_i$.
Essentially, $\Delta V$ has $O(\|R\|/|\lambda_i|)$ magnitude in the direction of $v_i$.
Crucially, such errors are suppressed when multiplied by $A$ as in~\eqref{eq:alte}, precisely by $\lambda_i$. Thus they cancel out to yield the quantity $\|R\|$ in our bounds for both $\|\Vh\Lh\Vh^T - V_+\Lambda_+V_+^T\|_F$ and $\|\Vh\Lh\Vh^T- V_+\Lambda_+V_+^T\|$ in~\eqref{eq:mainsimple} and~\eqref{eq:alte}, respectively. In what follows we make this intuition precise.

\emph{Notation}.
$\|\cdot\|_F$ denotes the Frobenius norm of a matrix, and $\|\cdot\|_2$ is the spectral norm (largest singular value).
$\|\cdot\|$ denotes a generic norm. We sometimes state (in)equalities that hold for any unitarily invariant norm; these will be stated explicitly.
$\Sym^n$ is the set of $n\times n$ Hermitian matrices, and $\Sym_+^n$ ($\Sym_-^n$)
is the set of positive (negative) semidefinite matrices in $\Sym^n$.
The projection operator onto the positive (negative) semidefinite cone is denoted $\project{+}$ ($\project{-})$; the domain of these operator should always be clear from the context. Given an orthonormal matrix $V$, we denote with $V_\perp$ some orthonormal matrix that spans the nullspace of $V^T$. $\lambda_{\max}(A)$ denotes the largest eigenvalue of a Hermitian matrix $A$ and $A \succ 0$ $(A \succeq 0)$ the positive (semi)definiteness of $A$.
For any matrix $B$, $\sigma_i(B)$ denotes the $i$th largest singular value.
Finally, given a set of vector $\{a_i\}$, $i \in \mathcal{I} \subseteq \mathbb{N}$, $[a_i]_{i \in \mathcal{I}}$, denotes horizontal concatenation of $a_{\mathcal{I}_{[1]}}, a_{\mathcal{I}_{[2]}}, \dots$ where $\mathcal{I}_{[i]}$ is the $i-$th smallest element in $\mathcal{I}$.

\ignore{
Let $A$ be an $n\times n$ symmetric matrix, and
$X\Lambda_+X^T$ its exact projection onto the semidefinite cone.
Let $k$ be the number of nonnegative eigenvalues, i.e., $\Lambda_+$ is
 $k\times k$.

Let $Q$ be an $n\times k$ matrix with orthonormal columns, such that
$Q^TAQ=\Lh$ is diagonal and positive semidefinite (this is obtained by extracting
nonnegative Ritz values (approximate eigenvalues) using a larger subspace containing $Q$).
Note that we are assuming that we have found exactly $k$ nonnegative Ritz values, equal to the number of nonnegative eigenvalues in the original matrix $A$ (note that the projected matrix $Q^TAQ$ cannot have more than $k$ nonnegative eigenvalues). \rr{We further assume that $(Q^\perp)^T AQ^\perp\preceq 0$, that is, the projection of $A$ onto the orthogonal complement $Q^\perp$ is negative semidefinite. This is a natural assumption, but does not necessarily follow from the assumption that $Q^TAQ$ has $k$ nonnegative eigenvalues.}
}

\section{Main result for the projection error}\label{sec:main}
In this section we derive the main result for the projection error $\|\Vh\Lh\Vh^T- V_+\Lambda_+V_+^T\|_F$. We begin with a presentation of the main idea used in this Section's proofs. Assume, for the purposes of this introductory presentation, that we have $\Lh = \Vh^T A \Vh$,  as it is common in practice, and that $\Lh \succeq 0$. Then, note that
\begin{equation}
  \begin{bmatrix}
    \Vh & \Vh_\perp
  \end{bmatrix}^T
  A
  \begin{bmatrix}
    \Vh & \Vh_\perp
  \end{bmatrix}
  =
  \begin{bmatrix}
    \Lh & \Rh^T \\
    \Rh & D
  \end{bmatrix}
\end{equation}
where $D \eqdef \Vh_\perp^TA\Vh_\perp$ and $\norm{{\Rh}} = \norm{R}$ for any unitarily invariant norm (derived in detail in Theorem \ref{thm:1}:\eqref{thm1:3}). Thus $\Ah \eqdef
  \begin{bmatrix}
    \Vh & \Vh_\perp
  \end{bmatrix}
  \begin{bmatrix}
    \Lh & 0 \\
    0 & D
  \end{bmatrix}
  \begin{bmatrix}
    \Vh & \Vh_\perp
  \end{bmatrix}^T
$
approximates $A$ with $\norm{A - \Ah} \leq \sqrt{2}\norm{R}$. Now, if we assume that $\Vh_\perp$ contains `most' of the negative eigenspace of $A$, in that $D \preceq 0$, then we have
$\project{+}(\Ah) = \Vh \Lh \Vh^T$.
% and $\Pi_-(\Ah) = \Vh_\perp^T D \Vh_\perp$
We can then write
\begin{equation}
  \|\Vh\Lh\Vh^T- V_+\Lambda_+V_+^T\| = \norm{\project{+}(\Ah) - \project{+}(A)}.
\end{equation}
The result~\eqref{eq:mainsimple} then follows immediately from the non-expansiveness of $\project{+}$ in the Frobenius norm \cite[Proposition 4.16]{Bauschke2011}, i.e., $\norm{\project{+}(\Ah) - \project{+}(A)}_F \leq \norm{\Ah - A}_F = \sqrt{2}\norm{R}_F$.

The following proofs generalize this result for the cases where $\Lh \neq \Vh^T A \Vh$ and/or $D \not \preceq 0$.
\ignore{
Below we work with the matrix
$    [\Vh \; \; \Vh_\perp]^T    A    [\Vh \; \; \Vh_\perp]  =\big[
\begin{smallmatrix}
\Vh^T A\Vh & R^T \\      R & D  \end{smallmatrix} \big]
$,
where
 $[\Vh\ \Vh_\perp]$ is orthogonal and
$  D =
\big[
\begin{smallmatrix}
D_+ & \\ & D_-
\end{smallmatrix}
\big]$ is diagonal (again, $\Vh^T A\Vh=\Lh$ is common but not necessary).
 $D_+$ has positive diagonals, which are ``correction'' terms that account for inexactness in $\Vh,\Lh$; $D_+$ is empty when $k=\widetilde k$ and
$R\rightarrow 0$.
%the case where the computed eigenvalues $\Lh$ did not capture the correct number of positive eigenvalues.
We also note that the use of $R$ is a slight abuse of notation, since it is not equal to the residual $A\Vh-\Vh\Lh$ used previously. This notation clash is left intentionally, since $\|R\|=\|A\Vh-\Vh\Lh\|$  in any unitarily invariant norm, and we only use the norm $\|R\|$ in the result.
}

\begin{theorem}\label{thm:1}

Suppose that the matrix $A\in \Sym^n$, $\Vh$ is an $n \times k$ orthonormal matrix with $k \le n$ and $\Lh$ is a $k\times k$ positive semidefinite matrix.
%Let $\Vh_\perp$ be such that $[\Vh \; \Vh_\perp]$ is orthogonal.
Then writing $R = A\Vh-\Vh\Lh$ and $D_+=\project{+}(\Vh_\perp\tpose A \Vh_\perp)$,
\begin{equation}\label{eqn:thm1}
\norm{\Vh\Lh \Vh\tpose - \project{+}(A)}_F^2 \le
\norm{R}_F^2
%\norm{A\Vh-\Vh\Lh}_F^2
 + \norm{\Vh_\perp\tpose A \Vh}_F^2 + \norm{D_+}_F^2.
\end{equation}

\end{theorem}

\begin{proof}
  Define
  \[
  B \eqdef \Vh\Lh \Vh\tpose + \Vh_\perp\project{-}(\Vh_\perp\tpose A\Vh_\perp)\Vh_\perp\tpose.
  \]

  Then
  \begin{equation}\label{thm1:1}
  \begin{aligned}
    \norm{\Vh\Lh \Vh\tpose - \project{+}(A)}^2_F &=
    \norm{\project{+}(B) - \project{+}(A)}^2_F \\
    &\le \norm{B - A}^2_F,
  \end{aligned}
\end{equation}
  where the inequality in the second line follows from the nonexpansiveness of the projection operator in the Frobenius norm~\cite[Proposition 4.16]{Bauschke2011}.  Since this norm is invariant with respect to unitary transformation, we have
\begin{equation}\label{thm1:2}
  \begin{aligned}
 \norm{B - A}^2_F &=
 \norm{\begin{bmatrix}\Vh\tpose \\ \Vh\tpose_\perp\end{bmatrix}(B - A)\begin{bmatrix}\Vh & \Vh_\perp\end{bmatrix}}^2_F \\
 &=
 \norm{
 \begin{bmatrix}
   \Lh - \Vh\tpose A\Vh & \Vh\tpose A \Vh_\perp \\
   \Vh_\perp\tpose A \Vh & - \project{+}(\Vh\tpose_\perp A \Vh_\perp)
 \end{bmatrix}
 }^2_F,
\end{aligned}
\end{equation}
  where the term in the lower right hand corner is formed using the identity $(I-\project{-}) = \project{+}$.  Considering next the norm of the residual $(A\Vh - \Vh\Lh)$ and applying a unitary transformation again, we have
  \begin{equation}\label{thm1:3}
\norm{R}_F^2=
\norm{A\Vh - \Vh\Lh}_F^2 =
 \norm{
 \begin{bmatrix}\Vh\tpose \\ \Vh\tpose_\perp\end{bmatrix}(A\Vh - \Vh\Lh)
 }_F^2 = \norm{\begin{bmatrix}\Vh\tpose A \Vh - \Lh \\ \Vh_\perp\tpose A \Vh\end{bmatrix}}_F^2.
\end{equation}
  The result then follows from combination of \eqref{thm1:1}, \eqref{thm1:2} and \eqref{thm1:3}.
\end{proof}

Theorem~\ref{thm:1} makes no assumption about the relationship of the matrix $A$ to the  matrix $\Vh\Lh \Vh\tpose$.   If we further assume that the latter matrix has been constructed from an approximation of $\project{+}(A)$ based on the Rayleigh-Ritz procedure, then we can go a bit further:

\begin{corollary}\label{cor:1}
Suppose that $(\Vh,\Lh)$ in Theorem~\ref{thm:1} satisfy the further relation
$
\Lh = \Vh\tpose A \Vh.
$
Then
\begin{equation}
\norm{\Vh\Lh \Vh\tpose - \project{+}(A)}_F^2 \le
%2\norm{A\Vh-\Vh\Lh}_F^2
2\norm{R}_F^2
 + \norm{D_+}_F^2. \label{eqn:cor1}
% + \norm{\project{+}(\Vh_\perp\tpose A \Vh_\perp)}_F^2.
\end{equation}
\end{corollary}
\begin{proof}
The top block in the rightmost expression in \eqref{thm1:3} becomes zero by assumption, leaving the relation $\norm{A\Vh-\Vh\Lh} =
\norm{\Vh_\perp\tpose A \Vh
}$ to be applied in \eqref{eqn:thm1}.
\end{proof}

The term $\norm{D_+}_F$ can be bounded by approximate computation of the largest eigenvalues of $(\Vh_\perp)^T    A\Vh_\perp$ by e.g.~Lanczos (in which instead of applying the unknown $\Vh_\perp$, we can apply $I-\Vh_{+}\Vh_+^T$). With a stable computation, we expect $\norm{D_+}$ to be very small. It is identically zero if $\Vh_\perp$ contains `most' of the negative eigenspace of $A$, in that $\Vh_\perp^TA\Vh_\perp\preceq 0$:
%spans a subset of the negative
\begin{corollary} \label{cor:main_result}
Suppose that the assumptions of Theorem~\ref{thm:1} and Corollary~\ref{cor:1} hold, and in addition $\Vh_\perp^TA\Vh_\perp\preceq 0$.
%$\image({V_\perp}) \subseteq \image (\project{-}(A))$.
Then
\begin{equation}
\norm{\Vh\Lh \Vh\tpose - \project{+}(A)}_F^2 \le
%2\norm{A\Vh-\Vh\Lh}_F^2.
2\norm{R}_F^2.
 \ \label{eqn:cor2}
\end{equation}
\end{corollary}
\begin{proof}
Obvious since the assumption about $\Vh_\perp$ zeros the term being projected in \eqref{eqn:cor1}.
\end{proof}

It is noteworthy how concise the proofs are---perhaps verging on appearing trivial.
The key fact enabling this is the nonexpansiveness of the operator $\project{+}$ (which itself is not trivial to establish~\cite[Proposition 4.16]{Bauschke2011}),
along with the introduction of the auxiliary matrix $B$.
We also note that the above proof provides little insight into why a small gap does not harm the bound.
In Section~\ref{sec:main2} we present a first-principle derivation
for bounding~\eqref{eq:alte}, which demonstrates clearly why the bounds are independent of the gap.

Before concluding this subsection we present an example that establishes sharpness of Corollary \ref{cor:main_result} up to a constant smaller than 2:
\begin{example} \label{ex:sharpness}
Consider
  \begin{equation}
    A = \begin{bmatrix}
       1 & -1 \\
      -1 &  0
    \end{bmatrix}, \quad
    \Vh = \begin{bmatrix}
      1 \\
      0
    \end{bmatrix}, \; \text{and} \quad
    \Lh = \Vh^T A \Vh = [1],
  \end{equation}
  in which $R= \begin{bmatrix}0 \\ -1\end{bmatrix}$, $\Vh \Lh \Vh^T = \begin{bmatrix} 1 & 0 \\ 0 & 0\end{bmatrix}$, $\project{+}(A) =
    \begin{bmatrix}
      \sqrt{5} + 3 & -\sqrt{5} - 1\\
      -\sqrt{5} - 1 & 2
    \end{bmatrix}/(2\sqrt{5}),$ and
  $$
  \norm{\Vh\Lh \Vh\tpose - \project{+}(A)}_F^2/\norm{R}_F^2 \approx 1.2764 > 1.
  $$
\end{example}

\subsection{When $A$ is nearly positive definite}
In practice, projection onto the semidefinite cone can be done in two ways:
\begin{enumerate}[i)]
\item compute the \emph{positive} eigenpairs $\Vh,\Lh$ such that $\Vh\Lh \Vh^T\approx V_+\Lambda_+V_+^T$ (which we implicitly assumed in the above arguments), or
\item
compute the \emph{negative} eigenpairs $\Vh_-,\Lh_-$ such that
$\Vh\Lh \Vh^T\approx V_-\Lambda_-V_-^T$, and obtain the approximate projection as $A-\Vh_-\Lh_-\Vh_-^T$.
\end{enumerate}
The former approach is conceptually more straightforward and is efficient when $A$ has a small number of positive eigenvalues. By contrast, the latter approach is much more efficient when $A$ is nearly positive definite, with the number of negative eigenvalues being small relative to the matrix size $n$.

We expect the second situation to be equally common if not more, and it is therefore important to derive analogous bounds applicable in case (ii). Fortunately, this is a trivial extension. We note that
\begin{align*}
&\norm{
(A-\Vh_-\Lh_-\Vh_-^T)
%\Vh\Lh \Vh\tpose
- \project{+}(A)}_F^2
=\norm{
(A-\Vh_-\Lh_-\Vh_-^T)
- (A-V_-\Lambda_-V_-^T)}_F^2 \\
=&\norm{\Vh_-\Lh_-\Vh_-^T-V_-\Lambda_-V_-^T}_F^2
=\norm{\Vh_-\Lh_-\Vh_-^T-(\project{-}(A))}_F^2\\
=&\norm{\Vh_-\Lh_-\Vh_-^T-(-\project{+}(-A))}_F^2,
\end{align*}
which is the accuracy of $\Vh_-\Lh_-\Vh_-^T$ as an approximate projection of the matrix $-A$ onto the semidefinite cone. We can therefore invoke the above results with $A\leftarrow -A$ to obtain the following.

\begin{corollary}
Suppose that the matrix $A\in \Sym^n$, $\Vh_-$ is a $n \times k$ orthonormal matrix with $k \le n$ and $\Lh$ is a $k \times k$ negative semidefinite matrix.
% Let $\Vh_{-,\perp} \in \Re^{n\times (n-k)}$ be such that $[\Vh_-\ \Vh_{-,\perp}]$ is orthogonal.
Then writing $R = A\Vh_--\Vh_-\Lh_-$ and $D_-=\project{-}(\Vh_{-,\perp}\tpose A \Vh_{-,\perp})$,
\begin{equation}\label{eqn:thm1b}
\norm{(A-\Vh_-\Lh_- \Vh_-\tpose) - \project{+}(A)}_F^2 \le
\norm{R}_F^2
%\norm{A\Vh-\Vh\Lh}_F^2
 + \norm{\Vh_{-,\perp}\tpose A \Vh_-}_F^2 + \norm{D_-}_F^2.
\end{equation}
If in addition we have
$\Lh_- = \Vh_-\tpose A \Vh_-$, then
\begin{equation}
\norm{(A-\Vh_-\Lh_- \Vh_-\tpose) - \project{+}(A)}_F^2 \le
%2\norm{A\Vh-\Vh\Lh}_F^2
2\norm{R}_F^2
 + \norm{D_-}_F^2, \label{eqn:cor1b}
% + \norm{\project{+}(\Vh_\perp\tpose A \Vh_\perp)}_F^2.
\end{equation}
and if furthermore $\Vh_{-,\perp}^T A\Vh_{-,\perp}\succeq 0$,
then
\begin{equation}
\norm{(A-\Vh_-\Lh_- \Vh_-\tpose) - \project{+}(A)}_F^2 \le
%2\norm{A\Vh-\Vh\Lh}_F^2.
2\norm{R}_F^2.
 \ \label{eqn:cor2b}
\end{equation}
\end{corollary}

\subsection{Extension to other norms} \label{sec:other_norms}
As we have already mentioned in the Introduction, a closest positive definite matrix to $A \in \mathbb{S}^n$ can be obtained by \eqref{eqn:projection} for any unitarily invariant norm as a distance metric. This is shown in the following Lemma:
\begin{lemma} \label{lem:projection}
  Given $A \in \mathbb{S}^n$ with eigendecomposition defined in \eqref{eqn:eigendecomposition}, $V_{+} \Lambda_{+} V_+^T$ is a solution to $\min_{X \succeq 0}\norm{A-X}$ for any unitarily invariant norm.
\end{lemma}
\begin{proof}
  Consider any $X \in \mathbb{S}^n_{+}$ and denote with $\alpha_1 \geq \dots \alpha_k \geq 0  \geq \alpha_{k+1} \geq \alpha_n$ the eigenvalues of $A$ and  $\chi_1 \geq \dots \chi_n \geq 0$ those of $X$.
  Then according to \cite[Corollary 7.4.9.3]{Horn2012} and following \cite[(7.4.9.2)]{Horn2012} we have:
\begin{align*}
  \norm{A - X} &\geq \norm{\diag(\alpha_1 - \chi_1, \dots, \alpha_n - \chi_n)} \\
  &\geq \norm{\diag(0, \dots, 0, \alpha_{k + 1} - \chi_{k + 1}, \dots, \alpha_n - \chi_n)} \\
  &\geq \norm{\diag(0, \dots, 0, \alpha_{k + 1}, \dots, \alpha_n)}.
\end{align*}
Thus $\norm{A - X} \geq \norm{\diag(0, \dots, 0, \alpha_{k}, \dots, \alpha_n)}$, with equality obtained in the final inequality by selecting $X = {V_{+} \Lambda_{+} V_+^T}$.
\end{proof}

Lemma \ref{lem:projection} has already been proven for the Frobenius and the spectral norm in the literature \cite{Higham1989}, but we are unaware of a proof that considers any unitarily invariant norm.

The projector to the semidefinite cone, i.e.~$\project{+}(A) \eqdef \text{argmin}_{X \succeq 0}\norm{A-X}$, is not unique in general (though it is for the Frobenius norm). Thus the projection error could be defined as $\min_{X \in \project{+}(A)}\norm{\Vt \Lt \Vt^T - X}_2$.
Nevertheless, for the rest of this section we consider the upper bound $\norm{\Vt \Lt \Vt^T - V_+ \Lambda_+ V_+^T}$ for reasons of simplicity and define $\project{+}(A)$ according to \eqref{eqn:projection}.

Perhaps surprisingly, the proof of Theorem~\ref{thm:1}
%(and the results)
does not carry over to other norms, including the spectral norm.
%\cite[Proposition 4.8]{Bauschke2011}
Specifically, the nonexpansiveness~\eqref{thm1:1}  of $\project{+}$
does not extend to %hold for
every unitarily invariant norm; simple computations reveal counterexamples for
e.g.~the spectral norm $\|A\|_2=\sigma_1(A)$ and trace norm $\|A\|_*=\sum_i\sigma_i(A)$.

Let us further investigate the spectral norm in particular. The issue is that for $n\geq 2$ there exist $A,B\in\mathbb{S}^n$ such that
 $\norm{\project{+}(A) - \project{+}(B)}_2 > \norm{A - B}_2$.
An example is $$A =
\begin{bmatrix}
M &1\\ 1 & 1/M
\end{bmatrix},
\quad
B =
\begin{bmatrix}
     M+1   &  0\\
     0   &  -1+1/M
\end{bmatrix},
$$
which as $M\rightarrow \infty$ gives $   \norm{\project{+}(A) - \project{+}(B)}_2/\norm{A - B}_2 \rightarrow \frac{\sqrt{5}+1}{2\sqrt{2}}\approx 1.1441$.
%\bb{YN thinks this might be the most extreme example with the largest possible ratio.}

The potential expansiveness of $\project{+}$ might be explained as follows. Let $[V_{B_+} \; V_{B_-}]$ be the orthogonal matrix of $B$'s eigenvectors such that
$$B = \begin{bmatrix}
  V_{B_+}\;\; V_{B_-}
\end{bmatrix}
\begin{bmatrix}
  \Lambda_{B_+} \\ &   \Lambda_{B_-}
\end{bmatrix}
\begin{bmatrix}
  V_{B_+}\;\; V_{B_-}
\end{bmatrix}^T,$$ where $\Lambda_{B_+}\!\succ 0, \Lambda_{B_-} \!\preceq 0$ are diagonal.
Then $\project{+}(B)
= [V_{B_+} \; V_{B_-}]
\big[
\begin{smallmatrix}
\Lambda_{B_+} \\ &  0
\end{smallmatrix}
\big][V_{B_+} \; V_{B_-}]^T$. Supposing $A \succeq 0$,
write $A = [V_{B_+} \; V_{B_-}]
\biggl[
{
\begin{smallmatrix}
\At_{11} & \At_{12} \\[1ex]
\At_{21} & \At_{22}
\end{smallmatrix}
}
\biggr][V_{B_+} \; V_{B_-}]^T$. Thus
\[
 \norm{\project{+}(A) - \project{+}(B)}_2
 = \norm{
   \begin{bmatrix}
\At_{11}-  \Lambda_{B_+} & \At_{12} \\
\At_{21} & \At_{22}
   \end{bmatrix}}_2,
\]
and potential expansiveness of $\project{+}$ means that this might be larger than $$\norm{A - B}_2 =
\norm{ \begin{bmatrix}
\At_{11}-  \Lambda_{B_+} & \At_{12} \\
\At_{21} & \At_{22} -\Lambda_{B_-}
   \end{bmatrix}}_2,$$ which is counterintuitive
as we clearly have $\|\At_{22} \|_2\leq \|\At_{22} -\Lambda_{B_-}\|_2$, as $\At_{22}$ and $-\Lambda_{B_-}$ are both positive semidefinite.
This %somewhat surprising
 fact is related to a classical result by Davis, Kahan and Weinberger~\cite{Davis1982} on norm-preserving dilation, which implies that
it is possible for the strict inequality
\begin{equation}
  \label{eqn:norm-dilation}
  \left\|
  \begin{bmatrix}
    X_{11}&  X_{12}\\X_{21}& X_{22}
    \end{bmatrix}\right\|_2<
  \left\|
    \begin{bmatrix}
    X_{11}&  X_{12}\\X_{21}& 0
    \end{bmatrix}
  \right\|_2
\end{equation}
to hold, even when the matrices are symmetric. For example, consider the case where $\At_{22}$  is negligible relative to $-\Lambda_{B_-}$. Then the above comparison essentially reduces to that of
$\norm{
\biggl[
\begin{smallmatrix}
\At_{11}-  \Lambda_{B_+} & \At_{12} \\
\At_{21} & 0
\end{smallmatrix}
\biggr]}_2$ and $  \norm{
\biggl[
\begin{smallmatrix}
\At_{11}-  \Lambda_{B_+} & \At_{12} \\
\At_{21} & -\Lambda_{B_-}
\end{smallmatrix}
\biggr]}_2$, which is in the form treated by Davis-Kahan-Weinberger.

Note that while
it is possible for \eqref{eqn:norm-dilation} to hold, we still have
$$\left\|
\begin{bmatrix}
   X_{11}&  X_{12}\\X_{21}& 0
\end{bmatrix}
\right\|_2<
\sqrt{\|[X_{11}\ X_{12}]\|_2^2 + \|[X_{21}\ 0]\|_2^2}
\leq
\sqrt{2}\left\|
\begin{bmatrix}
  X_{11}&  X_{12}\\X_{21}& X_{22}
\end{bmatrix}\right\|_2.$$
Similarly, for any $A \succeq 0$ we have
\begin{align*}
%\norm{\At - B}_2
 \norm{A - \project{+}(B)}_2
& = \norm{
   \begin{bmatrix}
\At_{11}-  \Lambda_{B_+} & \At_{12} \\
\At_{21} & \At_{22}
   \end{bmatrix}}_2\\
&\leq \sqrt{\|[\At_{11}-  \Lambda_{B_+} ,  \At_{12}]\|_2^2
+\|[\At_{21} , \At_{22}] \|_2^2}\\
&\leq \sqrt{\|[\At_{11}-  \Lambda_{B_+} ,  \At_{12}]\|_2^2
+\|[\At_{21} , \At_{22}-  \Lambda_{B_-}] \|_2^2}\\
&\leq
 \sqrt{2}
\norm{ \begin{bmatrix}
\At_{11}-  \Lambda_{B_+} & \At_{12} \\
\At_{21} & \At_{22} -\Lambda_{B_-}
   \end{bmatrix}}_2   = \sqrt{2}\|A - B\|_2.
\end{align*}
We conclude that $\norm{\project{+}(A) - \project{+}(B)}_2 \leq \sqrt{2}\norm{\project{+}(A) - B}_2$ for any $A, B \in \mathbb{S}^{n \times n}$, which shows quasi-nonexpansiveness of $\frac{1}{\sqrt{2}}\project{+}$ in the spectral norm \cite{Bauschke2011}. Unfortunately, this result is not sufficient to derive a bound like \eqref{eqn:cor2} using the techniques of Section \ref{sec:main}.

\ignore{
\hrule
We use this to derive bounds of the form
$   \norm{\project{+}(A) - \project{+}(B)}_2 \leq (1+\delta) \norm{A - B}_2$ for some $\delta>0$.
 By the triangle inequality (note that $A\succeq 0$ is not assumed)
\rr{Error here, final inequality does not hold since neither of $A,B$ is assumed posdef. probably remove between two hrules.}
\begin{equation}
  \label{eq:2normresult}
 \norm{\project{+}(A) - \project{+}(B)}_2 \leq  \norm{\project{+}(A) - B}_2 + \norm{B-A}_2+ \norm{A- \project{+}(B)}_2
\leq (1+2\sqrt{2})\norm{A-B}_2.
\end{equation}
This inequality is clearly loose and it is an open problem to obtain a tight bound
for $\delta$ such that $ \norm{\project{+}(A) - \project{+}(B)}_2 \leq (1+\delta)\norm{A-B}_2$ for any pair of symmetric matrices $A,B$.
On the other hand,~\eqref{eq:2normresult} still suffices to
yield (together with the proof of Theorem~\ref{thm:1}) a gap-free bound in the spectral norm

\begin{equation}
  \label{eqs:pec}
    \norm{\Pi_{+}(A) - \Vh \Lambda_+ \Vh^T}_2
    \leq
    (1+2\sqrt{2})(\norm{R}_2 + \norm{D_+}_2).
\end{equation}
\hrule
}

% As mentioned after Theorem~\ref{thm:projection_accuracy},
Nevertheless, based on experimental evidence we conjecture that, when $\Vh^T A \Vh \succeq 0$, we have
%the above bound~\eqref{eqs:pec} holds with $1+2\sqrt{2}$ replaced by $1$.
$    \norm{\Vh \Lh \Vh^T - \Pi_{+}(A)}_2^2
    \leq  \delta \norm{R}_2^2$ for a ``small'' constant $\delta$, perhaps $\delta = 2$. Note that the counterexample \ref{ex:sharpness} gives
    $$
      \norm{\Vh\Lh \Vh\tpose - \project{+}(A)}_2^2/\norm{R}_2^2 \approx 1.0935 > 1,
    $$
    thus the conjectured constant $\delta$ has to be larger than $1$.
%holds for any unitarily invariant norm (in particular spectral norm, which we discuss below), and that the constant $\sqrt{2}$ can be improved to $1$ (when $D_+$ is empty).

%\[
%%\norm{A - B}_2
%% = \norm{   \begin{bmatrix}     A_{11}-  \Lambda_{B_+} & A_{12} \\A_{21} & A_n{22}    \end{bmatrix}}_2.
%\]
%also shows that if $A\succeq 0$, then $   \norm{\project{+}(A) - \project{+}(B)}_2 \leq \sqrt{2} \norm{A - B}_2$.

%Consider for the moment (instead of $ \norm{\project{+}(A) - \project{+}(B)}_2 $)

%Using Theorem \ref{thm:projection_accuracy} we can see that Algorithm \ref{alg:approximate_projection} produces projections with errors that are summable.

\section{Bounding $\|A(V_+V_+^T-\Vh\Vh^T)\|_F  $}\label{sec:main2}
We now turn to the alternative measure~\eqref{eq:alte} for the projection accuracy.
First note that from $AV_+=V_+\Lambda_+$
and
$A\Vh=\Vh\Lh  + R$
we have
\[
AV_+V_+^T = V_+\Lambda_+ V_+^T, \qquad A\Vh\Vh^T=\Vh\Lh  \Vh^T + R\Vh^T,
\]
and hence
\[
%W =
 AV_+V_+^T-A\Vh\Vh^T = V_+\Lambda_+ V_+^T-\Vh\Lh  \Vh^T - R\Vh^T.
\]
Therefore we have
\begin{equation}  \label{eq:Wdef}
\|V\Lambda_+V^T-\Vh\Lh \Vh^T\|-
 \|R\|   \leq  \|A(V_+V_+^T-\Vh\Vh^T)\| \leq \|V\Lambda_+V^T-\Vh\Lh \Vh^T\|
%\|W\|
 +\|R\|.
\end{equation}

The two accuracy measures
$\|V\Lambda_+V^T-\Vh\Lh \Vh^T\|$ and
$\|A(V_+V_+^T-\Vh\Vh^T)\|$ are therefore at most $\|R\|$ apart; this immediately
gives the bound
$ \|A(V_+V_+^T-\Vh\Vh^T)\|_F   \leq    \sqrt{2\norm{R}_F^2 + \norm{D_+}_F^2}+\norm{R}_F$
 as a corollary of~Theorem~\ref{thm:1}.

Here we follow a different argument to directly bound $\|A(V_+V_+^T-\Vh\Vh^T)\|$,
producing a tighter result.
% in terms of $\|R\|_F$.
%a neater result with a smaller constant.
Moreover, while the proof is longer than in Theorem~\ref{thm:1},
it clearly reveals how the gap-independence comes about.

In what follows we assume $\Lh $ is obtained by the Rayleigh-Ritz process, i.e., $\Lh =\Vh^TA\Vh$. Furthermore, we define $[\lambda_1\; \dots \; \lambda_n]^T \eqdef [\diag(\Lambda_+)^T\; \diag(\Lambda_-)^T]^T$ and $[v_1 \; \dots v_n] \eqdef [V_+ \; V_-]$.
% In view of this, to bound $\|W\|$ we attempt to bound $\|A(V_+V_+^T-\Vh\Vh^T)\|$.
\begin{theorem}
Under the notation and assumptions in Theorem~\ref{thm:1},
\begin{equation}
  \label{eq:VVTbound}
  \|A(V_+V_+^T-\Vh\Vh^T)\|_F^2\leq
2\|R\|_F^2+2 \|R\|_F\|D_+\|_F+ \|D_+\|_F^2
\leq 2(\|R\|_F + \|D_+\|_F)^2.
%\leq 2\|R\|_F + \|D_+\|_F.
%\|A(V_+V_+^T-\Vh\Vh^T)\|_F\leq
%\sqrt{2\|R\|_F^2+2 \|R\|_F\|D_+\|_F+ \|D_+\|_F^2}\leq 2\|R\|_F + \|D_+\|_F
\end{equation}
\end{theorem}
\begin{proof}
Following the arguments in~\cite[\S 2.5.3]{Golub2013} we have
for any unitarily invariant norm
\begin{align}
\|A(V_+V_+^T-\Vh\Vh^T)\|&=\left\|[V_+ \; V_-]
\begin{bmatrix}
\Lambda_+ \\ & \Lambda_-
\end{bmatrix}\bigl[V_+ \;\; V_-\bigr]^T
 (V_+V_+^T-\Vh\Vh^T)\bigl[\Vh \;\; \Vh_\perp\bigr]\right\| \nonumber \\
&=\left\|
\begin{bmatrix}
\Lambda_+ &0 \\
0 & \Lambda_-
\end{bmatrix}
\begin{bmatrix}
0&  V_+^T\Vh_\perp \\
(V_-)^T\Vh&0
\end{bmatrix}\right\|. \label{eq:AXXQQ}
\end{align}
Let us examine the size of the $i$th row of the $(1,2)$ block of the rightmost matrix term, i.e.\ $\|v_i^T\Vh_\perp\|_2$. This is precisely the sine of the angle between $\Vh$ and $v_i$. Assume, without loss of generality, that $\Vh_\perp$ was chosen such that $A\Vh_\perp=\Vh_\perp D + \Rh$ where
%$\Lh^\perp\prec 0$
$  D =
\big[
\begin{smallmatrix}
D_+ & \\ & D_-
\end{smallmatrix}
\big]$
and
$  \Dt =
\big[
\begin{smallmatrix}
0 & \\ & D_-
\end{smallmatrix}
\big]$,
is diagonal
% as in Theorem~\ref{thm:1},
%(looks like we need to assume this; this is natural but does not necessarily follow from our $k=k$ assumption)
and $\|\Rh\|= \|R\|$ in any unitarily invariant norm. Then,
\[
v_i^TA\Vh_\perp=\lambda_iv_i^T\Vh_\perp
= v_i^T\Vh_\perp D + v_i^T\Rh
= v_i^T\Vh_\perp \Dt + v_i^T\Vh_\perp
\begin{bmatrix}
  D_+ & \\& 0
\end{bmatrix}
%D_+
 + v_i^T\Rh.
\]
%where $\Vh = [\Vh_{+,+},\Vh_{+,-}]$.
%$\Vh = [\Vh_{+,+},\Vh_{+,-}]$.
Hence we have
\[
v_i^T\Vh_\perp(\lambda_iI-\Dt) = v_i^T(\Vh_\perp
\begin{bmatrix}
D_+ & \\ & 0
\end{bmatrix}
 + \Rh ),
\]
thus  $v_i^T\Vh_\perp=
v_i^T(\Vh_\perp
\big[\begin{smallmatrix}
D_+ & \\ & 0
\end{smallmatrix}\big]
 + \Rh)(\lambda_iI-\Dt)^{-1} $ and the inverse is guaranteed to exist since $\lambda_i$ is necessarily positive. Therefore,
%v_i^T\Rh(D_--\lambda_iI)^{-1}$, so
using the fact $\Dt\preceq 0$ we obtain
\begin{equation}  \label{eq:xiQ}
\|v_i^T\Vh_{\perp}\|_2
\leq \frac{\|v_i^T(\Vh_\perp \big[\begin{smallmatrix}
D_+ & \\ & 0
\end{smallmatrix}\big]
+\Rh)\|_2}{\min_j\Bigl|\lambda_i-\Dt_{j,j}\Bigr|}
\leq \frac{\|v_i^T(\Vh_\perp \big[\begin{smallmatrix}
D_+ & \\ & 0
\end{smallmatrix}\big]
+\Rh)\|_2}{\lambda_i}.
\end{equation}

%Hence Now each row of $X^TQ^\perp$ is bounded in norm by $\|v_i^TR\|_2/\lambda_i$ (for $i=1,\ldots,k$),
Since this holds for $i=1,\ldots,k$, it follows that the $i$th row of $\Lambda_+ V_+^T\Vh_\perp$ is bounded by
$\|v_i^T(\Vh_\perp \big[\begin{smallmatrix}
D_+ & \\ & 0
\end{smallmatrix}\big]
+\Rh)\|_2$.
%\leq \|v_i^T\Vh_+^\perp D_+\|_2+\|v_i^T\Rh\|_2$.
 Note how the $\Lambda_{+,i,i}=\lambda_i$ and
$1/\lambda_i$ terms neatly cancel out; this is why the result is gap-independent.
Together with the bounds for $i=1,\ldots,k$, we obtain
\begin{equation}  \label{eq:12term}
\|\Lambda_+ V_+^T\Vh_\perp\|_F\leq
\|V_+^T(\Vh_\perp
\big[\begin{smallmatrix}
D_+ & \\ & 0
\end{smallmatrix}\big]+\Rh)\|_F\leq \|\Vh_\perp \big[\begin{smallmatrix}
D_+ & \\ & 0
\end{smallmatrix}\big]+\Rh\|_F\leq \|D_+\|_F+\|\Rh\|_F,
\end{equation}
where we used the fact that $V_+$ %and $\Vh_{+,2}^\perp$
has orthonormal columns, hence $\|V_+^TX\|_F\leq \|X\|_F$ for any matrix $X$.
%It follows that the (1,2) block $ V_+^T\Vh_\perp$ of~\eqref{eq:AXXQQ} has Frobenius norm bounded by $\| D_+\|_F+\|\Rh\|_F$

Similarly, we examine
the $i$th row of $V_-^T\Vh$, which appears
in the $(2,1)$ block of~\eqref{eq:AXXQQ}.
From $A\Vh=\Vh\Lh+  R$ we have
\[
v_{k+i}^TA\Vh=
\lambda_{k+i}v_{k+i}^T\Vh
= v_{k+i}^T(\Vh\Lh +  R).
%= v_{k+i}^T(\Vh\Dt + \Vh\big[\begin{smallmatrix}D_+ & \\ & 0  \end{smallmatrix}\big]+ +  R).
\]
%Hence defining $\Dt \equiv  \begin{bmatrix}0 & \\ & D_-   \end{bmatrix}$ as in~\eqref{eq:dtilde},
%$\lt_{k+i}:=\min(0,\lambda_{k+i})$,
Hence we have
\[
v_{k+i}^T\Vh (\lambda_{k+i}I-\Lh)=
v_{k+i}^TR,
%  -\max(0,\lambda_{k+i}) v_{k+i}^T\Vh.
\]
so $v_{k+i}^T\Vh=v_{k+i}^TR(\lambda_{k+i}I-\Lh)^{-1}$, giving
\[
\|v_{k+i}^T\Vh\|\leq \frac{\|v_{k+i}^TR\|}{\min_j\Bigl|\lambda_{k+i}-\Lt_{+,j,j}\Bigr|}\leq \frac{\|v_{k+i}^TR\|}{\|\lambda_{k+i}|}.
\]
Here we have used the facts that  $\Lh\succ 0$ and $\lambda_{k+i}\leq 0$;
we assumed $\lambda_{k+i}< 0$, which can be done because the terms with $\lambda_{k+i}=0$ do not contribute to the terms in~\eqref{eq:AXXQQ}.
%each row of $(X^\perp)^TQ$ is bounded by $\|R\|_2/|\lambda_i|$ (for $i=k+1,\ldots,n$; note $\lambda_i<0$). It therefore follows that
Since this holds for $i=k+1,\ldots,n$,
the $(k+i)$th row of \eqref{eq:AXXQQ} is bounded in norm by
$\|v_{k+i}^TR\|_2$. Thus
%$(2,1)$ block can be bounded by
\begin{equation}  \label{eq:21}
\|\Lambda_-V_-^T\Vh\|_F\leq \|R\|_F  .
\end{equation}

%every row in the above expression~\eqref{eq:AXXQQ}
%It follows that for every $i\in\{1,\ldots,n\}$, the $i$th row of~\eqref{eq:AXXQQ} has norm bounded by $\|v_i^TR\|_2$. We therefore see that the Frobenius norm of the entire matrix is bounded as $\|A(V_+V_+^T-\Vh\Vh^T)\|_F\leq \|R\|_F$, as required.
Putting ~\eqref{eq:AXXQQ},~\eqref{eq:12term} and \eqref{eq:21}  together, we obtain
\[
\|A(V_+V_+^T-\Vh\Vh^T)\|_F\leq
\sqrt{2\|R\|_F^2+2 \|R\|_F\|D_+\|_F+ \|D_+\|_F^2},
\]
as required.
It is easy to see the final expression is upper bounded by $\sqrt{2}(\|R\|_F + \|D_+\|_F)$.
%[\bb{YN}: the old note had an error: $V_+^T\Rh$ and $(V_+^\perp)^TR$ do not lie in orthogonal complements so we cannot say $\|V_+^T\Rh\|_F^2+\|(V_+^\perp)^TR\|_F^2 = \|R\|_F^2$;  indeed $\|V_+^T\Rh\|_F=\|(V_+^\perp)^TR\|_F=\|R\|_F$ by construction].
\end{proof}

\ignore{
Together with \eqref{eq:Wdef} we conclude that
\begin{equation}
  \label{eq:conclusionF}
\|V_+\Lambda V_+^T-\Vh\Lh \Vh^T  \|_F \leq 2\|R\|_F.
\end{equation}
}

%Note that for the spectral norm, we have the bound $\|V_+\Lambda V_+^T-\Vh\Lh \Vh^T  \|_2 \leq (1+\sqrt{k})\|R\|_2$---which is better than the trivial bound $(1+\sqrt{n})\|R\|_2$---because the nonzero blocks in \eqref{eq:AXXQQ} have rank bounded by $k$.
%(gap independent, yay! But the constant $\sqrt{n}$ is terrible. )

\section{Experiments}
In this section we demonstrate the results of this paper with a set of experiments. We begin with an experiment that demonstrates the independence of the projection accuracy on the spectral gap. We consider a set of parametric matrices
$A \in \mathbb{S}^{500}$ with $20$ eigenvalues geometrically distributed on $[10^{-10}, 1]$, one at $\epsilon$, where $\epsilon$ is a parameter which is used to control the spectral gap, one at $-\epsilon$ and the rest uniformly distributed on $[-1, 0]$. We compute approximate projections by running \texttt{ARPACK} \cite{Lehoucq1998} (accessed via \texttt{MATLAB}'s \texttt{eigs} with its default parameters) on $A$ requesting the $25$ largest eigenpairs and returning $\Vh \Lh \Vh^T$, where $(\Vh, \Lh)$ are the positive Ritz pairs obtained by \texttt{ARPACK}. We let \texttt{ARPACK} iterate until convergence (set to its default value, $10^{-14}$) and we do this for varying values of $\epsilon$. In Figure \ref{fig:gap_independence} we plot the resulting projection errors as a function of the spectral gap. We observe that the projection accuracy is not impaired by small spectral gaps, and that the bound of Corollary \ref{cor:1} successfully estimates the high accuracy of the projection.
\begin{figure}
	\centering
	\includegraphics[width=.49\textwidth]{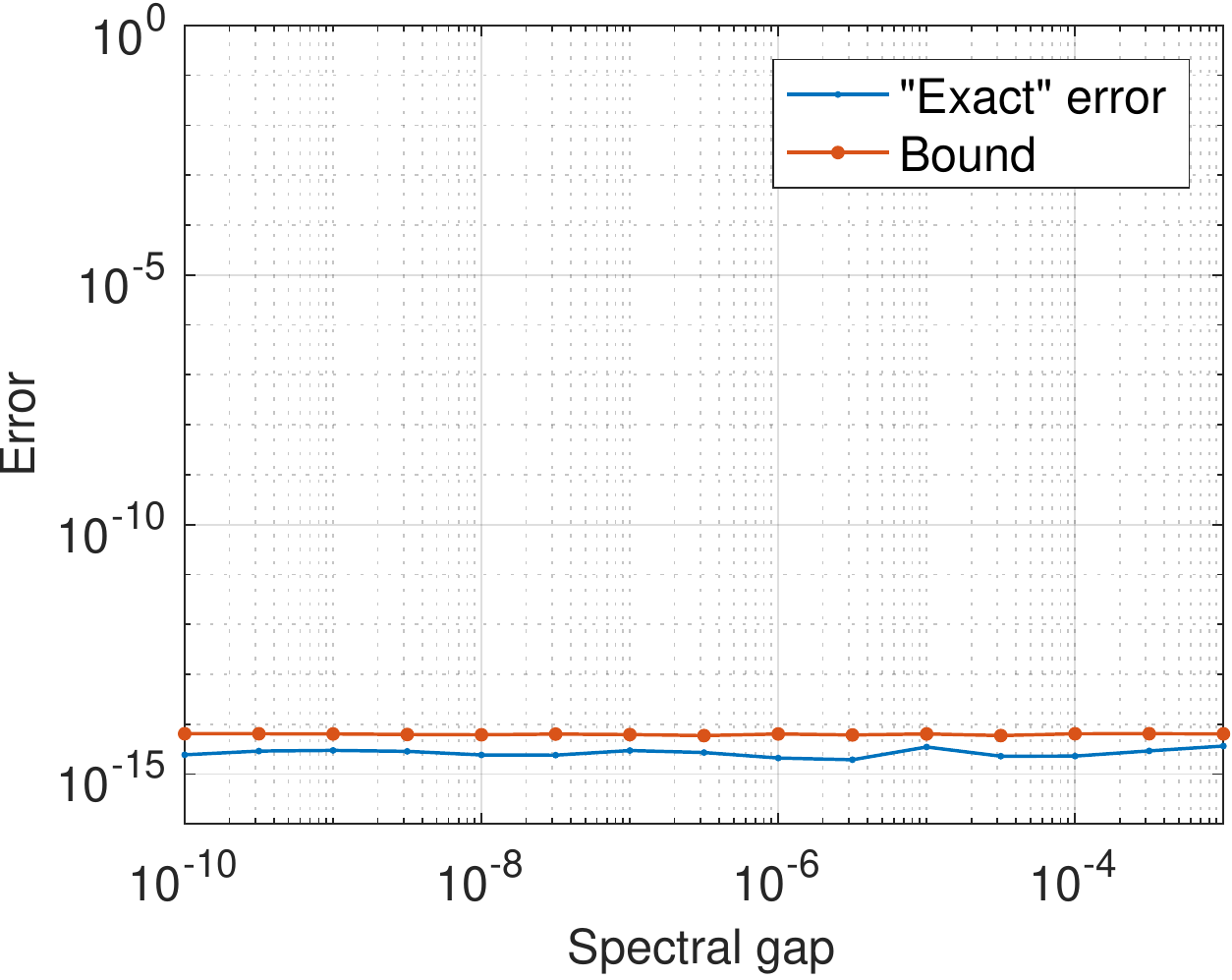}
  \caption{Left: ``Exact'' projection error $\norm{V_+ \Lambda_+ V_+^T - \Vh \Lh \Vh^T}_F$, where $V_+, \Lambda_+$ are obtained via a full eigenvalue decomposition and $\Vh, \Lh$ are computed with \texttt{ARPACK}, and the respective bound obtained via Corollary \ref{cor:1} for problems with varying spectral gap.}
  \label{fig:gap_independence}
\end{figure}

Next, we consider the accuracy of the projection for the iterates produced by \texttt{ARPACK}. We consider a matrix $A \in \mathbb{S}^{500}$ with eigenvalues distributed according to the previous experiment with $\epsilon = 10^{-10}$. We compute approximate projections by running \texttt{ARPACK} as before. Figure \ref{fig:convergence} (left) shows the projection accuracy as a function of \texttt{ARPACK}'s iterations. The upper bound of Corollary \ref{cor:1} is compared with the ``exact'' projection error $\norm{V_+ \Lambda_+ V_+^T - \Vh \Lh \Vh^T}_F$ where $V_+, \Lambda_+$ are obtained via a full eigenvalue decomposition. Although our bound eventually approximates the exact error, it exhibits oscillatory behaviour before convergence.
\begin{figure}
	\centering
	\begin{minipage}[t]{.49\textwidth}
	\includegraphics[width=.99\textwidth]{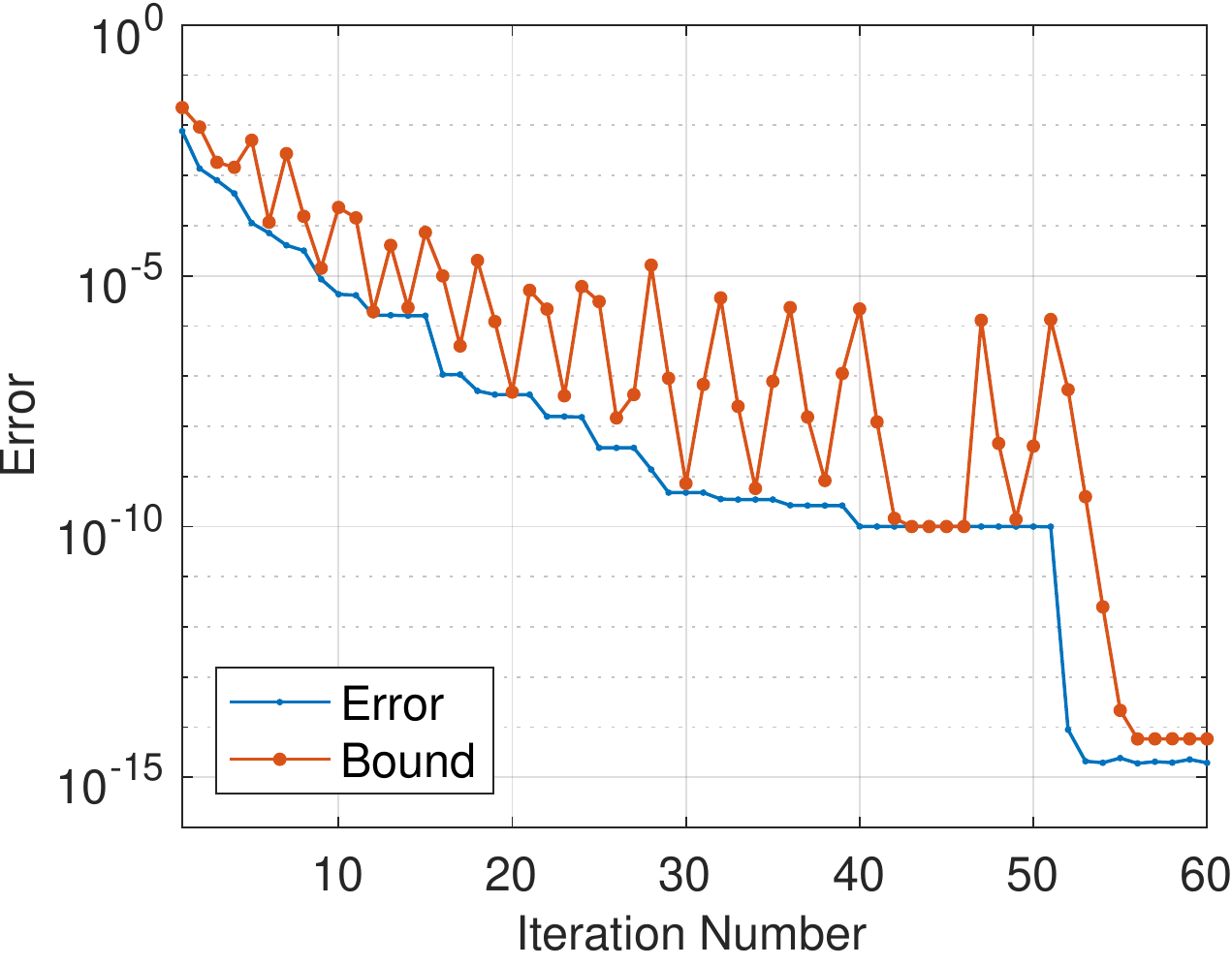}
	\end{minipage}
	\centering
	\begin{minipage}[t]{.49\textwidth}
	\includegraphics[width=.99\textwidth]{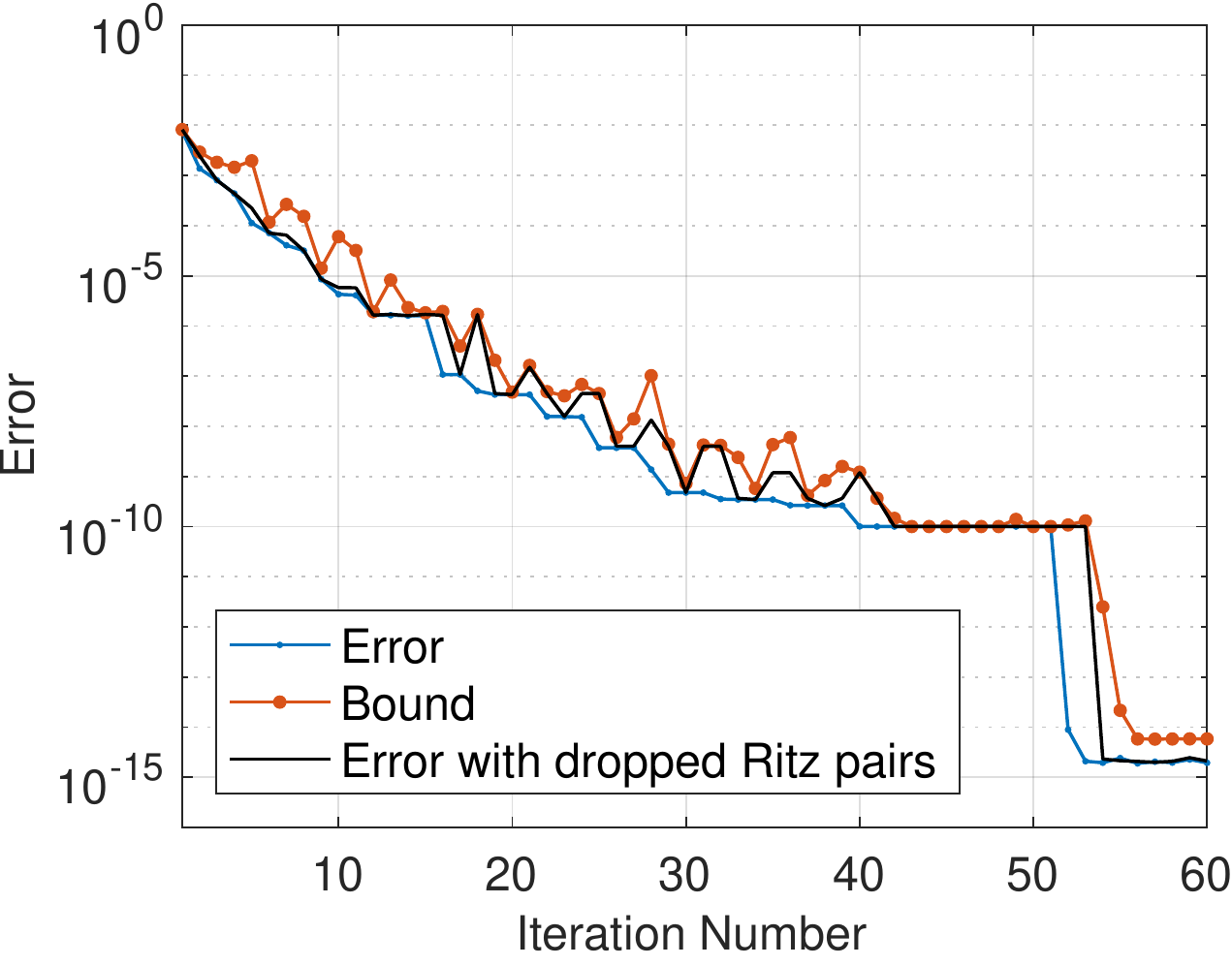}
	\end{minipage}
  \caption{Left: Convergence of the ``exact'' projection error $\norm{V_+ \Lambda_+ V_+^T - \Vh \Lh \Vh^T}_F$, where $V_+, \Lambda_+$ are obtained via a full eigenvalue decomposition and $\Vh, \Lh$ are computed with \texttt{ARPACK}, and the respective bound obtained via Corollary \ref{cor:1}.
  Right: Convergence of $\norm{V_+ \Lambda_+ V_+^T - \Vh \Lh \Vh^T}_F$,
  $\norm{V_+ \Lambda_+ V_+^T - \Vt \Lt \Vt^T}_F$
  (black solid) and the respective bound obtained by applying Corollary \ref{cor:1} for $\Vt \Lt \Vt^T$.
  }
  \label{fig:convergence}
\end{figure}

We will show that this oscillatory behaviour is caused by positive Ritz pairs with large residuals. The following lemma provides a criterion for excluding some positive Ritz pairs resulting in improved error bounds, which as we will see, exhibit significantly reduced oscillatory behaviour.

\begin{lemma} \label{lem:drop_ritzpairs}
  Suppose that $\Lh = diag(\hat \lambda_1, \dots, \hat \lambda_k) \succeq 0$ and $\Vh = [\hat v_1 \; \dots \; \hat v_k]$ are a set of Ritz pairs for some $A \in \mathbb{S}^n$. Furthermore, consider any $\Vt = [\hat v_i]_{i \not \in  \mathcal{I}}$ $\Lt = \diag([\hat \lambda_i ]_{i \not \in \mathcal{I}})$ where $\mathcal{I} \subseteq \{1, \dots, k\}$ with every $i \in \mathcal{I}$ satisfying
  \begin{equation} \label{eqn:residual_condition}
    (\sqrt{2} - 1) \norm{r_i}_2 > \max\left(
      \hat \lambda_i, \lambda_{\max}(\Vh_\perp^T A \Vh_\perp)
      \right),
  \end{equation}
  with $r_i \eqdef A\hat v_i - \hat \lambda_i\hat v_i$.
  Then, the bound of Corollary \ref{cor:1} for the projection error $$\norm{\Vt \Lt \Vt^T - V_+ \Lambda_+ V_+^T}_F$$ is smaller (i.e.~better) than the respective bound of $\norm{\Vh \Lh \Vh^T - V_+ \Lambda_+ V_+^T}_F$.
\end{lemma}

\begin{proof}
  We will first prove the case where $\mathcal{I}$ contains a single index and then generalize for the case where $\mathcal{I}$ has multiples elements.

  Without loss of generality, assume that $\mathcal{I} = \{ 1 \}$. Define $[\hat v_1\; \Vh_2] \eqdef \Vh$,
$\begin{bmatrix}
  \hat \lambda_1 & 0 \\
  0 & \Lh_2
\end{bmatrix} \eqdef \Lh$, and $[r_1 \; R_2] \eqdef R$.
Corollary \ref{cor:1} gives the following bound for the approximate projection
$
\Vh \Lh \Vh^T = \hat \lambda_1 \hat v_1 \hat v_1^T + \Vh_2 \Lh_2 \Vh_2^T
$:
\begin{align}
  \label{eqn:error}
  \norm{\Vh \Lh \Vh^T - V_+\Lambda_+ V_+^T}_F^2
  \leq \underbrace{2 \norm{R_2}_F^2 + 2 \norm{r_1}_2^2 + \norm{\project{+}(D)}_F^2}_{\eqdef \text{bound}_1}
\end{align}
where $D \eqdef \Vh_\perp^T A \Vh_\perp$.
%where $D_+ = \project{+}(\Vh_\perp^T A \Vh_\perp)$
If we do not include $\hat \lambda_1 \hat v_1 \hat v_1^T$ in our approximate projection, then we obtain the following bound:
\begin{align*}
  %\text{error}_2
  %\eqdef
  \norm{\Vh_2 \Lh_2 \Vh_2^T -  V_+\Lambda_+ V_+^T}_F
  \leq
  \underbrace{
    2\norm{R_2}^2_F + \norm{\project{+}(\tilde D)}_F^2
  }_{\eqdef \text{bound}_2}
\end{align*}
where $\tilde D \eqdef [\hat v_1 \; \Vh_\perp]^T A [\hat v_1 \; \Vh_\perp]$. Note that
\begin{equation} \label{eqn:bound_diff}
  \text{bound}_2 - \text{bound}_1 = - 2\norm{r_1}_2^2 + \norm{\project{+}(\tilde D)}_F^2 - \norm{\project{+}(D)}_F^2
\end{equation}
and denote with $\mu_1\leq \dots \leq \mu_{k+1}$ the eigenvalues of $D$ and $\tilde \mu_1 \leq \dots \leq \tilde \mu_{k}$ those of $\tilde D$. Then, using \cite[Theorem 10.1.1]{Parlett1998} on $
 \tilde D =
 \begin{bmatrix}
  \hat \lambda_1 & \hat v_1^T A \Vh_\perp \\
  \Vh_\perp^T A \hat v_1 & \Vh_\perp^T A \Vh_\perp
 \end{bmatrix}
 $ we get $\tilde \mu_i \leq \mu_i$ where $i = 1, \dots k$. Thus,
\begin{align}
\norm{\project{+}(\tilde D)}_F^2 - \norm{\project{+}(D)}_F^2 &= \sum_{i=1}^{k+1} {\max}^2(\tilde\mu_i, 0) - \sum_{i=1}^k {\max}^2(\mu_i, 0) \\
% &= {\max}^2(\tilde \mu_{l+1}, 0) + \sum_{i=1}^{l} {\max}^2(\tilde\mu_i, 0) - {\max}^2(\mu_i, 0) \\
\label{eqn:remainder_diff}
&\leq {\max}^2(\tilde \mu_{k+1}, 0). % = \norm{\project{+}(\tilde D)}_2
\end{align}
Furthermore, using \cite[Theorem 10.3.1]{Parlett1998} on $
\tilde D =
\begin{bmatrix}
  \hat \lambda_1 & 0 \\
  0 & D
\end{bmatrix}
+
\begin{bmatrix}
0 & \hat v_1^T A \Vh_\perp \\
\Vh_\perp^T A \hat v_1 & 0
\end{bmatrix}
$ we get
\begin{equation} \label{eqn:diff_max_eigenvalues}
  \tilde \mu_{l+1} \leq \max(\mu_l, \hat \lambda_1) + \norm{
    \begin{bmatrix}
      0 & \hat v_1^T A \Vh_\perp \\
      \Vh_\perp^T A \hat v_1 & 0
    \end{bmatrix}}_2
  = \max(\mu_l, \hat \lambda_1) + \norm{r}_2,
\end{equation}
where the last equality holds because $\Vh$ and $\Lh$ were obtained by performing the Rayleigh-Ritz process on $A$ (see \eqref{thm1:3}).

Combining \eqref{eqn:bound_diff}, \eqref{eqn:remainder_diff} and \eqref{eqn:diff_max_eigenvalues} gives
\begin{equation} \label{eqn:bound_inequality}
  \text{bound}_2 \leq \text{bound}_1 \quad \text{when} \quad \max(\mu_l, \hat \lambda) \leq \left(\sqrt{2} - 1\right)\norm{r}_2.
\end{equation}
This completes the proof for the case where $\mathcal{I}$ is a singleton.

Consider now the case where $\mathcal{I} = \{i_1, \dots, i_l\}$ with $l > 1$. According to the proceeding part of this proof, the Ritz pairs
$([\hat \lambda_i ]_{i \neq i_1},[\hat v_i]_{i \neq i_1})$ will produce an approximate projection with a lower approximation bound (obtained via Corollary \ref{cor:1}) than those of $\Vh \Lh \Vh^T$. Likewise, $([\hat \lambda_i ]_{i \neq i_1, i_2},[\hat v_i]_{i \neq i_1, i_2})$ will produce an approximate projection with a further improved bound if
\begin{equation} \label{eqn:residual_second_condition}
  (\sqrt{2} - 1) \norm{r_{i_2}}_2 > \max\left(
    \hat \lambda_{i_2}, \lambda_{\max}(W^T A W)
  \right).
\end{equation}
where $W$ is an orthonormal matrix spanning the nullspace of $[\hat v_i]_{i \neq i_1}^T$. However, \eqref{eqn:residual_second_condition} is implied by \eqref{eqn:residual_condition} because $\text{span}(W) \subset \text{span}(\Vh_\perp)$ and thus $\lambda_{\max}(W^T A W) \leq \lambda_{\max}(\Vh_\perp^T A \Vh_\perp)$ \cite[Corollary 4.1]{Saad2011}. Since the same argument holds for the rest of the indices contained in $\mathcal{I}$, this concludes the proof.
\end{proof}

% Thus, given a set of positive Ritz vector $\Vh = [\hat v_1, \dots, \hat v_k]$ and Ritz values $\hat \Lambda = \diag(\hat \lambda_1, \dots, \hat \lambda_k)$ a practical criterion for to obtain an approximate projection with improved bounds (as compared to the ones obtained by applying Corollary \ref{cor:1} for $\Vh \Lh \Vh^T$) is to discard all the Ritz pairs $\{(\hat v_i, \hat \lambda_i)\}$ where
%$$
%\norm{r_i}_2 \geq \frac{1}{\sqrt{2} - 1}\max\left(\hat \lambda_i, \lambda_{\max}(\Vh_\perp^T A \Vh_\perp)\right)
%$$
Figure \ref{fig:convergence} (right) aims to demonstrate the usefulness of Lemma \ref{lem:drop_ritzpairs} in reducing the oscillations of Figure \ref{fig:convergence} (left). It considers the same experiment as Figure \ref{fig:convergence} (left), but compares the convergence of the ``exact'' accuracy of $\norm{V_+ \Lambda_+ V_+^T -  \Vt \Lt \Vt^T}_F$ where $( \Lt, \Vt)$ consists of the positive Ritz pairs contained in $(\Lh, \Vh)$ for which $(\sqrt{2} - 1) \norm{r_i}_2 \leq \max\left(
\hat \lambda_i, \lambda_{\max}(\Vh_\perp^T A \Vh_\perp)
\right)$, with the respective bound obtained by applying Corollary \ref{cor:1}. We observe a significant reduction in the oscillatory behaviour of our bound as compared to Figure \ref{fig:convergence} (left). We further include in Figure \ref{fig:convergence} (right) the convergence of $\norm{V_+ \Lambda_+ V_+^T - \Vh \Lh \Vh^T}_F$ for comparison. It is worth noting that, in this experiment, $\norm{V_+ \Lambda_+ V_+^T - \Vh \Lh \Vh^T}_F$ converges almost monotonically and it is never greater than either of $\norm{V_+ \Lambda_+ V_+^T - \Vt  \Lt \Vt^T}_F$ or our bound.

\bibliographystyle{abbrv}\bibliography{../Bibliography/bibliography}

\end{document}